\documentclass[11pt]{article}
\usepackage{amssymb,amsmath,amsthm,hyperref,verbatim,fullpage,youngtab}

\newtheorem{claim}{Claim}
\newtheorem{theorem}{Theorem}
\newtheorem{lemma}[theorem]{Lemma}

\newtheorem{conjecture}[theorem]{Conjecture}

\newtheorem{corollary}[theorem]{Corollary}

\newcommand{\A}{\mathcal{A}}
\newcommand{\M}{\mathcal{M}}
\newcommand{\C}{\mathcal{C}}
\newcommand{\G}{\mathcal{G}}

\newcommand{\Bin}{\textrm{Bin}}
\usepackage{color}
\newcommand{\todo}[1]{\textcolor{red}{#1}}
\newcommand{\abs}[1]{\left|{#1}\right|}
\newcommand{\ceil}[1]{\left\lceil{#1}\right\rceil}

\newcommand{\numberthis}{\addtocounter{equation}{1}\tag{\theequation}}

\begin{document}
\title{Intersections of iterated shadows}
\author{Hou Tin Chau\footnote{School of Mathematics, University of Bristol. Research supported by an EPSRC Doctoral Training Studentship.}, David Ellis\footnote{School of Mathematics, University of Bristol.}\hspace{3pt} and Marius Tiba\footnote{Department of Mathematics, King's College London.}}
\date{September 2024}
\maketitle
\begin{abstract}
We show that if $\mathcal{A} \subset {[n] \choose n/2}$ with measure bounded away from zero and from one, then the $\Omega(\sqrt{n})$-iterated upper shadows of $\mathcal{A}$ and $\mathcal{A}^c$ intersect in a set of positive measure. This confirms (in a strong form) a conjecture of Friedgut. It can be seen as a stability result for the Kruskal--Katona theorem.
\end{abstract}

\section{Introduction}
For $n \in \mathbb{N}$, we write $[n]: = \{1,2,\ldots,n\}$ for the standard $n$-element set; the power-set $\mathcal{P}([n])$ is known as the {\em $n$-dimensional discrete cube}, or {\em $n$-cube} for short. For $1 \leq k \leq n$, we write ${[n] \choose k}: = \{A \subset [n]:\ |A|=k\}$ for the set of all $k$-element subsets of $[n]$, known as the {\em $k$th layer} of the $n$-cube. If $1 \leq k \leq n$ and $\mathcal{A} \subset {[n] \choose k}$, we write 
$$\partial^{+}(\mathcal{A}): = \left\{B \in {[n] \choose k+1}:\ B \supset A \text{ for some } A \in \mathcal{A}\right\},$$
for the upper shadow of $\mathcal{A}$, and we write
$$\partial^{+r}(\mathcal{A}) = \left\{B \in {[n] \choose k+r}:\ B \supset A \text{ for some } A \in \mathcal{A}\right\}$$
for its $r$th iterate, for any $r \leq n-k$.

Now let $n$ be even, and let $\mathcal{A} \subset {[n] \choose n/2}$. The Kruskal--Katona theorem \cite{kruskal,katona} supplies a sharp bound on how small $|\partial^{+}(\mathcal{A})|$ can be, in terms of $|\mathcal{A}|$; equality holds if $\mathcal{A}$ is an initial segment of the lexicographical ordering on ${[n] \choose n/2}$. Since the upper shadow of an initial segment of the lexicographic ordering on ${[n] \choose k}$ is an initial segment of the lexicographic ordering on ${[n] \choose k+1}$ (for any $k$), applying the Kruskal--Katona theorem iteratively yields a sharp bound on how small $|\partial^{+r}(\mathcal{A})|$ can be, in terms of $|\mathcal{A}|$. Letting $\mu$ denote the uniform measure on the appropriate layer of the $n$-cube, i.e.
$$\mu(\mathcal{F}): = |\mathcal{F}|/{n \choose k}$$
if $\mathcal{F} \subset {[n] \choose k}$, the Kruskal--Katona theorem implies that if $\mathcal{A} \subset {[n] \choose k}$ with $\mu(\mathcal{A}) = 1/2$, then $\mu(\partial^{+r}(\mathcal{A})) \geq 1/2+r/n$; equality holds if $\mathcal{A} = \{A \in {[n] \choose n/2}:\ 1 \in A\}$. Interestingly however, if $\mathcal{A} = \{A \in {[n] \choose n/2}:\ 1 \in A\}$, then $\mathcal{A}^c = \{A \in {[n] \choose n/2}:\ 1 \notin A\}$ and $\mu(\partial^{+}(\mathcal{A}^c)) = 1$, i.e., one only needs to take the upper shadow of $\mathcal{A}^c$ (once, without iterating) to obtain everything. In general, it seems impossible to come up with examples of half-sized sets $\mathcal{A} \subset {[n] \choose n/2}$ where $\partial^{+r}(\mathcal{A})$ and $\partial^{+r}(\mathcal{A}^c)$ do not have intersection of positive measure, if $r$ is at least (any) linear function of $\sqrt{n}$. Considering the example $\mathcal{A} = \{A \in {[n] \choose n/2}:\ |A \cap [n/2]| > n/4\}$, where $n$ is congruent to two modulo four, shows that we may have $\mu(\mathcal{A})=1/2$ and 
$$\mu(\partial^{+r}(\mathcal{A})\cap\partial^{+r}(\mathcal{A}^c)) = O(r/\sqrt{n}),$$
so the hypothesis that $r$ grows at least linearly in $\sqrt{n}$ is necessary to obtain the conclusion above (that $\partial^{+r}(\mathcal{A})$ and $\partial^{+r}(\mathcal{A}^c)$ intersect in a set of positive measure). Friedgut (personal communication) conjectured that $r = \Omega(\sqrt{n})$ is sufficient; to be precise, he made the following conjecture.
\begin{conjecture}[Friedgut, 2024, personal communication]
For any $\zeta >0$ and $\epsilon >0$, there exists $\delta >0$ such that the following holds. If $n$ is even and $\mathcal{A} \subset {[n] \choose n/2}$ with $\zeta \leq \mu(\mathcal{A})\leq 1-\zeta$ then, for $r = \lceil \epsilon \sqrt{n} \rceil$, we have
    $$\mu(\partial^{+r}(\mathcal{A})\cap\partial^{+r}(\mathcal{A}^c)) \geq \delta.$$
\end{conjecture}
Our purpose in this note is to prove Friedgut's conjecture, in the following quantitative form.

\begin{theorem}
\label{thm:main}
    There exists an absolute constant $C>0$ such that the following holds. Let $\epsilon \in (0,1)$, let $n$ be even and let $\mathcal{A} \subset {[n] \choose n/2}$ with $\A \neq \emptyset$ and $A \neq {[n] \choose n/2}$. Suppose further that
    $$n \geq \frac{C}{\epsilon^2}\left(\log(1/\epsilon)+\log\left(\frac{1}{\mu(\A)(1-\mu(\A))}\right)\right).$$
    Then, for $r = \lceil \epsilon \sqrt{n} \rceil$, we have
    $$\mu(\partial^{+r}(\mathcal{A})\cap\partial^{+r}(\mathcal{A}^c)) = \Omega\left(\frac{\epsilon^2}{\log \left(\frac{1}{\epsilon \mu(\A)(1-\mu(\A))}\right)}\right)\mu(\A)(1-\mu(\A)).$$
\end{theorem}
This can be seen as a stability result for the Kruskal--Katona theorem, since it implies that if $\mathcal{A} \subset {[n] \choose n/2}$ with $\mu(\mathcal{A}) = 1/2$, then provided $n$ is sufficiently large depending on $\epsilon$, writing $r: = \lceil \epsilon \sqrt{n} \rceil$ we have
\begin{align*}
\mu(\partial^{+r}(\mathcal{A}))+\mu(\partial^{+r}(\mathcal{A}^c))& = \mu(\partial^{+r}(\mathcal{A})\cup \partial^{+r}(\mathcal{A}^c))+\mu(\partial^{+r}(\mathcal{A})\cap\partial^{+r}(\mathcal{A}^c))\\&= 1+\mu(\partial^{+r}(\mathcal{A})\cap\partial^{+r}(\mathcal{A}^c))\\
&\geq 1+\delta(\epsilon),
\end{align*}
where $\delta(\epsilon): = c\epsilon^2/\log(1/\epsilon)$ for some absolute constant $c>0$, so
$$\max\{\mu(\partial^{+r}(\mathcal{A})),\mu(\partial^{+r}(\mathcal{A}^c))\} \geq 1/2+\delta(\epsilon)/2,$$
whereas a naive application of Kruskal--Katona just yields $\mu(\partial^{+r}(\mathcal{A})) \geq 1/2+\epsilon/\sqrt{n}$ (and also, of course, $\mu(\partial^{+r}(\mathcal{A}^c)) \geq 1/2+\epsilon/\sqrt{n}$). In other words, at least one of $\mathcal{A}$ and $\mathcal{A}^c$ must have large iterated upper-shadow.

We note that the explicit lower bound on $n$ in the statement of Theorem \ref{thm:main} could be replaced by the following (more usual, but less explicit) formulation: for any $\zeta >0$ and any $\epsilon \in (0,1)$, there exists $n_0 = n_0(\zeta,\epsilon)>0$ such that if $\zeta \leq \mu(\A)\leq 1-\zeta$ and $n \geq n_0$, the conclusion of the theorem holds. 

Our proof has two main ingredients. First, we observe that if we had $r = \lceil \sqrt{n}\rceil$ instead of $r = \lceil \epsilon \sqrt{n}\rceil$, then the required lower bound on $\mu(\partial^{+r}(\mathcal{A})\cap\partial^{+r}(\mathcal{A}^c))$ would follow quickly from Harris' inequality. The second ingredient is to use a random restriction argument, and the expansion properties of a certain class of Johnson graphs, to reduce the case $r = \lceil \epsilon \sqrt{n}\rceil$ to the case $r = \lceil \sqrt{n}\rceil$. For $1 \leq k \leq n$ and $1 \leq j \leq k$, we write $J(n,k,j)$ for the graph with vertex-set ${[n] \choose k}$, where two $k$-element sets are joined by an edge of the graph $J(n,k,j)$ iff their symmetric difference has size $2j$. The expansion result we require is that for $n$ even and $1 \leq j \leq n/10$, the nonzero eigenvalues of the normalized Laplacian of $J(n,n/2,j)$ are all at least $\Omega(j/n)$ --- i.e., that the normalized spectral gap of $J(n,n/2,j)$ is $\Omega(j/n)$. This result does not seem to be stated in the (extensive) literature on Johnson graphs, though it is not very hard to prove it, with the aid of a technique recently introduced by Koshelev \cite{koshelev}.

\section{Proof of Theorem \ref{thm:main}.}

We first prove the required result on the spectral gap of various Johnson graphs.

\begin{theorem}
\label{thm:spect}
    Let $n$ be even and let $j \in \mathbb{N}$ with $\eta: = j/n \leq 1/10$. Let $\tilde{L}$ be the normalised Laplacian of the Johnson graph $H: = J(n,n/2,j)$, i.e., $\tilde{L} = L/d = I-A/d$, where $I$ is the ${n \choose n/2}$ by ${n \choose n/2}$ identity matrix, $L$ is the (unnormalised) Laplacian of $H$, $A$ is the adjacency matrix of $H$, and $d = {n/2 \choose j}^2$ is the degree of $H$ (meaning, $H$ is $d$-regular). Let $0 = \tilde{\mu}_1 < \tilde{\mu}_2 \leq \ldots \leq \tilde{\mu}_{{n \choose n/2}}$ denote the eigenvalues of $\tilde{L}$, in increasing order, counted with multiplicity. Then we have
    $$\tilde{\mu}_2 \geq \eta/2.$$
\end{theorem}
\begin{proof}
It is well-known (see e.g.\ \cite[p. 48]{Delsarte} and \cite[Theorem 2.1]{Karloff}) that the eigenvalues of $A$ may be enumerated (not with multiplicity) as $\lambda(0),\lambda(1),\lambda(2),\ldots,\lambda(k)$, where
$$\lambda(i) = \sum_{h=0}^{\min\{i,j\}}(-1)^h {i \choose h}{k-i \choose j-h}^2,$$
and $k: = n/2$; for any real number $\lambda$, the $\lambda$-eigenspace of $A$ is
$$\bigoplus_{i:\lambda_i = \lambda} W_i,$$
where
$$W_i: = V_i \cap V_{i-1}^{\perp}$$
for all $i \geq 1$, $W_0: = V_0$ and for each $0 \leq i \leq k$, $V_i \leq \mathbb{R}[{[n] \choose n/2}]$ is the linear subspace of $\mathbb{R}[{[n] \choose n/2}]$ spanned by the {\em $i$-juntas}, i.e., functions depending on at most $i$ coordinates. For example, $V_0$ is the linear subspace of constant functions, and $V_1 = \text{Span}\{\phi_j:\ 1 \leq j \leq n\}$, where
$$\phi_j:\left\{x \in \{0,1\}^n:\ \sum_{i=1}^{n}x_i=k\right\} \to \mathbb{R};\quad x \mapsto x_j.$$
In general, for $0 \leq i \leq k$, we have $V_i = \text{Span}\{\phi_S:\ S \in {[n] \choose i}\}$, where
$$\phi_S: \left\{x \in \{0,1\}^n:\ \sum_{i=1}^{n}x_i=k\right\} \to \mathbb{R};\quad x \mapsto \prod_{j \in S} x_j$$
for each $S \in {[n] \choose i}$.

The eigenvalue of $\tilde{L}$ corresponding to $\lambda(i)$ is
$$\tilde{\mu}(i) := 1-\lambda(i)/d = 1-\lambda(i)/{k \choose j}^2.$$
For $1 \leq i \leq j$, we have
$$\tilde{\lambda}(i):=\lambda(i)/{k \choose j}^2 = \sum_{h=0}^{i}(-1)^h {i \choose h}\left({k-i \choose j-h}/{k \choose j}\right)^2.$$

It suffices to show that for all even $n$, all $i\in [n/2]$ and all $\eta \le 1/10$, we have $\tilde{\lambda}(i) \le (1-\eta/2)^i$. We use the trick on p.5 of \cite{koshelev}. 
\begin{align*}
    \abs{\lambda(i)/{k \choose j}^2} 
&= \abs{\sum_{h=0}^{i}(-1)^h {i \choose h}\left({k-i \choose j-h}/{k \choose j}\right)^2} \\
&\le  \max_{h=0}^i {k-i \choose j-h} \cdot \sum_{h=0}^{i} {i \choose h} {k-i \choose j-h}/{k \choose j}^2 \\
&= \max_{h=0}^i {k-i \choose j-h} /{k \choose j}. \\
\end{align*}
Case 1: if $\ceil{ (k-i)/2} < j\ (\le k/2)$, then $i > k - 2j = (1/2 - 2\eta) n$ and the last expression is 
\begin{align*}
\max_{h=0}^i {k-i \choose j-h} /{k \choose j} 
&\le {k-i \choose \ceil{(k-i)/2}} /{k \choose j} \\
&\le \frac{2^{k-i}}{(k/j)^j} \\
&= 2^{k-i} \cdot (2\eta)^j \\
&= 2^{-i} \cdot \left(2^{1/2} \cdot (2\eta)^\eta\right)^n \\
&\le 2^{-i} \cdot \left((2 - \eta)^{1/2 - 2\eta}\right)^n &\text{ note that here we use } \eta \le 1/10 \\
&\le 2^{-i} \cdot (2-\eta)^i \\
&= (1-\eta/2)^i.
\end{align*}
Case 2: if $j \le \ceil{(k-i)/2}$, then the expression is 
\begin{align*}
\max_{h=0}^i {k-i \choose j-h} /{k \choose j} 
&= {k-i \choose j} /{k \choose j} \\
&\le \left(\frac{k-i}{k}\right)^j\\
&= \left(\frac{n/2 - i}{n/2}\right)^{\eta n} \\
&\le \exp\left(-\frac{2i}{n}\cdot \eta n\right)  \\
&= \exp\left(-2i\eta\right) \\
&\le (1 - \eta)^i.
\end{align*}
\end{proof}

From this, and the discrete Cheeger inequality of Alon and Milman \cite{am}, we immediately deduce the expansion consequence we need.

\begin{corollary}
\label{cor:pair-count}
If $\mathcal{A} \subset {[n] \choose n/2}$, and $j$, $H$ and $d$ are defined as in Theorem \ref{thm:spect}, then we have
\begin{align*}\left|\left\{(A,B) \in {[n] \choose n/2}^2:\ |A \Delta B| = 2j,\ A \in \mathcal{A},\ B \notin \mathcal{A}\right\}\right| & = e_H(\mathcal{A},\mathcal{A}^c) \\
& \geq d\ \tilde{\mu}_2\ \mu(\mathcal{A})(1-\mu(\mathcal{A})){n \choose n/2}\\
&\geq \tfrac{1}{2}\eta\ \mu(\mathcal{A})(1-\mu(\mathcal{A}))d{n \choose n/2},
\end{align*}
and therefore
\begin{align*}\frac{|\{(A,B) \in {[n] \choose n/2}^2:\ |A \Delta B| = 2j,\ A \in \mathcal{A},\ B \notin \mathcal{A}\}|}{|\{(A,B) \in {[n] \choose n/2}^2:\ |A \Delta B| = 2j\}|} &= \frac{|\{(A,B) \in {[n] \choose n/2}^2:\ |A \Delta B| = 2j,\ A \in \mathcal{A},\ B \notin \mathcal{A}\}|}{d{n \choose n/2}} \\
& \geq \tfrac{1}{2}\eta\ \mu(\mathcal{A})(1-\mu(\mathcal{A})).
\end{align*}
\end{corollary}

We are now in a position to prove Theorem \ref{thm:main}. 
    \begin{proof}[Proof of Theorem \ref{thm:main}.]
Let $n$ be even, let $\mathcal{A} \subset {[n] \choose n/2}$ with $A \neq \emptyset$ and $\A \neq {[n] \choose n/2}$, and let $j \in \mathbb{N}$ be defined as in Theorem \ref{thm:spect}, where $\eta$ is to be chosen later. 

Define the set of {\em good pairs} to be
$$\left\{(A,B) \in {[n] \choose n/2}^2:\ |A \Delta B| = 2j,\ A \in \mathcal{A},\ B \notin \mathcal{A}\right\},$$
and the set of {\em distance-$(2j)$ pairs} to be
$$\left\{(A,B) \in {[n] \choose n/2}^2:\ |A \Delta B| = 2j\right\}$$
note that these are precisely the pairs counted in Corollary \ref{cor:pair-count}, which says precisely that if a distance-$(2j)$ pair is chosen uniformly at random, it is good with probability 
$$\mathfrak{q} \geq \tfrac{1}{2} \eta \mu(\A)(1-\mu(\A)).$$

Let $D = D(n)$ be an even integer to be chosen later, with $D \geq 2j$, and let $\mathcal{C}$ be a $D$-dimensional subcube of the $n$-cube whose top point is a single set of size $n/2+D/2$ and whose bottom point is a single set of size $n/2-D/2$; note that the middle layer of $\mathcal{C}$ is a subset of the middle layer ${[n] \choose n/2}$ of the whole $n$-cube. Let $\mathcal{M}(\mathcal{C}) = \mathcal{C} \cap {[n] \choose n/2}$ denote the middle layer of $\mathcal{C}$; note that $|\mathcal{M}(\mathcal{C})| = {D \choose D/2}$. Define
$$\alpha_\mathcal{C} : = |\mathcal{A} \cap \mathcal{M}(\mathcal{C})|/|\mathcal{M}(\mathcal{C})|$$
to be the measure of the intersection of $\mathcal{A}$ with $\mathcal{M}(\mathcal{C})$, the middle-layer of $\mathcal{C}$. Let $\gamma_\C$ denote the proportion of distance-$(2j)$ pairs in $\M(\C)$ that are good. Observe that, if a distance-$(2j)$ pair is chosen uniformly at random from $\M(\C)$, then the first element of the pair is simply a uniform random element of $\M(\C)$; if the pair is good, its first element must be in $\mathcal{A}$, so the probability that the pair is good (which is $\gamma_{\C}$) is at most the probability that its first element is in $\mathcal{A}$ (which is $\alpha_{\C}$). Hence, $\gamma_\C \leq \alpha_\C$. The same argument applied to the second element of the pair, yields $\gamma_\C \leq 1-\alpha_\C$. Therefore,
$$\alpha_\C(1-\alpha_\C) \geq \tfrac{1}{2}\min\{\alpha_\C,1-\alpha_\C\} \geq \tfrac{1}{2}\gamma_\C.$$
Let $\mathcal{U}_{\mathcal{C}}(\mathcal{A} \cap \mathcal{C})$ denote the up-set in $\mathcal{C}$ generated by $\mathcal{A} \cap \mathcal{C}$ --- i.e.,
$$\mathcal{U}_{\mathcal{C}}(\mathcal{A} \cap \mathcal{C}) = \{X \in \C:\ X \supset A \text{ for some }A \in \A \cap \C\} \subset \mathcal{C}.$$
Similarly, we let
$\mathcal{U}_{\mathcal{C}}(\mathcal{A}^c \cap \mathcal{C})$ denote the up-set in $\mathcal{C}$ generated by $\mathcal{A}^c \cap \mathcal{C}$ --- i.e.,
$$\mathcal{U}_{\mathcal{C}}(\mathcal{A}^c \cap \mathcal{C}) = \{X \in \C:\ X \supset B \text{ for some }B \in \A^c \cap \C\} \subset \mathcal{C}.$$
Let $\mu_\C$ denote the uniform measure on the subcube $\C$, which is a discrete cube of dimension $D$. By the local LYM inequality applied within $\C$, the set $\mathcal{U}_{\mathcal{C}}(\mathcal{A} \cap \mathcal{C})$ occupies at least an $\alpha_\C$ proportion of every layer of $\C$ from its middle layer and upwards, and therefore
$$\mu_\C(\mathcal{U}_{\mathcal{C}}(\mathcal{A} \cap \mathcal{C})) \geq \tfrac{1}{2}\alpha_\C.$$
Similarly,
$$\mu_\C(\mathcal{U}_{\mathcal{C}}(\mathcal{A}^c \cap \mathcal{C})) \geq \tfrac{1}{2}(1-\alpha_\C).$$
By Harris' inequality \cite{harris}, applied to the up-sets $\mathcal{U}_{\mathcal{C}}(\mathcal{A} \cap \mathcal{C})$ and $\mathcal{U}_{\mathcal{C}}(\mathcal{A}^c \cap \mathcal{C})$ within $\C$, we have
$$\mu_\C(\mathcal{U}_{\mathcal{C}}(\mathcal{A} \cap \mathcal{C}) \cap \mathcal{U}_{\mathcal{C}}(\mathcal{A}^c \cap \mathcal{C})) \geq \mu_\C(\mathcal{U}_{\mathcal{C}}(\mathcal{A} \cap \mathcal{C}))\mu_\C(\mathcal{U}_{\mathcal{C}}(\mathcal{A}^c \cap \mathcal{C})) \geq \tfrac{1}{4}\alpha_\C(1-\alpha_\C) \geq \tfrac{1}{8}\gamma_\C.$$
Consider the set
$$\G_\C: = \{X \in \mathcal{U}_{\mathcal{C}}(\mathcal{A} \cap \mathcal{C}) \cap \mathcal{U}_{\mathcal{C}}(\mathcal{A}^c \cap \mathcal{C}):\ |X| \leq n/2+\lceil K\sqrt{D}\rceil\}.$$
By a Chernoff bound, we have $\mu_\C(\{X \in \C: |X| > n/2+\lceil K\sqrt{D} \rceil\}) = \Pr[\Bin(D,1/2) > D/2+\lceil K\sqrt{D} \rceil] < \exp(-2K^2/3)$, and therefore
$$\mu_\C(\G_\C) \geq \mu_\C(\mathcal{U}_{\mathcal{C}}(\mathcal{A} \cap \mathcal{C}) \cap \mathcal{U}_{\mathcal{C}}(\mathcal{A}^c \cap \mathcal{C})) - \exp(-2K^2/3) \geq \tfrac{1}{8}\gamma_\C - \exp(-2K^2/3).$$
Let $\ell = D/2+\lceil K\sqrt{D}\rceil$ and let $\mu_{\C_\ell}$ denote the uniform measure on the $\ell$th layer $\C_\ell$ of $\C$, which is a subset of the $n/2+\lceil K\sqrt{D}\rceil$ layer of the whole $n$-cube. Since $\mathcal{U}_{\mathcal{C}}(\mathcal{A} \cap \mathcal{C}) \cap \mathcal{U}_{\mathcal{C}}(\mathcal{A}^c \cap \mathcal{C})$ is an up-set within $\C$, we have
$$\mu_{\C_\ell}(\G_\C \cap \C_\ell) \geq \tfrac{1}{8}\gamma_\C - \exp(-2K^2/3).$$
Now choose $\C$ uniformly at random (over all $D$-dimensional subcubes of the $n$-cube whose top point is a set of size $n/2+D/2$ and whose bottom point is a single set of size $n/2-D/2$). Let $r = \lceil K\sqrt{D}\rceil$. A uniform random element of the $n/2+r$ layer of the $n$-cube is precisely a uniform random element of the $\ell$th layer of a uniform random $\C$, and since $\partial^{+r}(\mathcal{A})\cap\partial^{+r}(\mathcal{A}^c) \supset \G_\C \cap \C_\ell$ for each $\C$, we have
$$\mu(\partial^{+r}(\mathcal{A})\cap\partial^{+r}(\mathcal{A}^c)) \geq \mathbb{E}_\C[\mu_{\C_\ell}(\G_\C \cap \C_\ell)] \geq \tfrac{1}{8}\mathbb{E}_\C[\gamma_\C] - \exp(-2K^2/3).$$
Since a uniform random distance-$(2j)$-pair within ${[n] \choose n/2}$ is precisely a uniform random distance-$(2j)$-pair within $\C$ for a uniform random $\C$, we have
$$\mathbb{E}_\C[\gamma_\C] = \mathfrak{q} \geq \tfrac{1}{2}\eta \mu(\A)(1-\mu(\A)),$$
and therefore
$$\mu(\partial^{+r}(\mathcal{A})\cap\partial^{+r}(\mathcal{A}^c)) \geq \tfrac{1}{16}\eta \mu(\A)(1-\mu(\A)) - \exp(-2K^2/3).$$
We must now choose $K$ and $D$ such that $K\sqrt{D} = \epsilon \sqrt{n}$; we also pick $D = 2j = 2\lceil \eta n \rceil$. This forces $2K^2 = \epsilon^2 n /\lceil \eta n \rceil \geq \epsilon^2/(2\eta)$ if $\eta n \geq 1/2$ (which we shall guarantee). To optimize (up to constant factors) our lower bound, we pick $\eta$ to satisfy 
$$\tfrac{1}{32} \eta \mu(\A) (1-\mu(\A)) = \exp(-\epsilon^2/(3\eta)),$$
which also guarantees that $2\eta \leq 1/10$ (and therefore that $\lceil \eta n \rceil /n \leq 2\eta \leq 1/10$). With this choice of $\eta$, we have $$\eta = \Omega\left(\frac{\epsilon^2}{\log \left(\frac{1}{\epsilon \mu(\A)(1-\mu(\A))}\right)}\right)$$ and therefore, provided
$$\frac{c\epsilon^2}{\log \left(\frac{1}{\epsilon \mu(\A)(1-\mu(\A))}\right)}\geq \frac{1}{2n}$$
for an appropriate absolute constant $c>0$, we obtain
$$\mu(\partial^{+r}(\mathcal{A})\cap\partial^{+r}(\mathcal{A}^c)) = \Omega\left(\frac{\epsilon^2}{\log 
\left(\frac{1}{\epsilon \mu(\A)(1-\mu(\A))}\right)}\right)\mu(\A)(1-\mu(\A)),$$
proving the theorem.
\end{proof}

\section{Conclusion and open problems}
We believe that, for half-sized sets, the example $\mathcal{A} = \{A \in {[n] \choose n/2}:\ |A \cap [n/2]| > n/4\}$ mentioned above is extremal for Friedgut's problem, and therefore we make the following conjecture.
\begin{conjecture}
\label{conj:main}
  If $n$ is even and $\mathcal{A} \subset {[n] \choose n/2}$ with $\mu(\mathcal{A}) = 1/2$ then for all $\epsilon>0$, if $r = \lceil \epsilon \sqrt{n} \rceil$ then
    $$\mu(\partial^{+r}(\mathcal{A})\cap\partial^{+r}(\mathcal{A}^c)) = \Omega(\epsilon).$$  
\end{conjecture}
Our random restriction argument above necessarily creates a quadratic dependence upon $\epsilon$. It seems that new techniques will be required to prove Conjecture \ref{conj:main}, if it is true.

\subsubsection*{Acknowledgements}
We are very grateful to Ehud Friedgut for posing the problem considered here, and to Dor Minzer for pointing us to known results on the spectral gap of Johnson graphs.


\begin{thebibliography}{99}
\bibitem{am} N. Alon and V. D. Milman. $\lambda_1$, isoperimetric inequalities for graphs, and superconcentrators. {\em J. Combin. Theory, Ser. B} 38 (1985), 73--88.
\bibitem{Delsarte} P. Delsarte. An algebraic approach to the association schemes of coding
theory. {\em Philips Res. Rep. Suppl.} 10 (1973), vi. 97 pp.
\bibitem{harris} T. E. Harris. A lower bound for the critical probability in a certain percolation process. {\em Math. Proc. Cam. Phil. Soc.} 56 (1960), 13--20.
\bibitem{katona} G. O. H. Katona. A theorem of finite sets. In P. Erd\H{o}s and G. O. H. Katona (Eds), {\em Theory of Graphs} (Proc. Colloq., Tihany, 1966), pp.\ 187--207.
Academic Press, New York, 1968.
\bibitem{koshelev} M. Koshelev, Spectrum of Johnson graphs. {\em Discrete Math.}  346 (2023), 113262.
\bibitem{Karloff} H. Karloff. How good is the Goemans-Williamson max cut algorithm? {\em SIAM J. Comput.} 20 (1999), 336–-350.
\bibitem{kruskal} J. B. Kruskal. The number of simplices in a complex. In R.E. Bellman (Ed), {\em Mathematical Optimization Techniques}, pp.\ 251--278. University of California Press, Berkeley, 1963.
\end{thebibliography}
\end{document}